\documentclass[12pt]{amsart}
\usepackage{amsmath,amsthm,amsfonts,amssymb,latexsym}
\usepackage{tikz-cd}
\usepackage{hyperref}
\usepackage{ulem}
\usepackage{pgf,tikz}
\usepackage{mathrsfs}
\usetikzlibrary{arrows}

\usepackage{enumerate}
\usepackage[shortlabels]{enumitem}
\usepackage{color}
\usepackage{xcolor} 
\begin{document}

\theoremstyle{plain}

\newtheorem{thm}{Theorem}[section]
\newtheorem*{ThmA}{Theorem A}
\newtheorem*{ThmB}{Theorem B}
\newtheorem*{ThmC}{Theorem C}
\newtheorem*{ThmD}{Theorem D}
\newtheorem*{PropB}{Proposition B}
\newtheorem*{CorB}{Corollary B}
\newtheorem*{CorC}{Corollary C}
\newtheorem*{ConjectureB}{Conjecture B}
\newtheorem{rem}[thm]{Remark}
\newtheorem{corollary}[thm]{Corollary}
\newtheorem{lemma}[thm]{Lemma}
\newtheorem{theorem}[thm]{Theorem}
\newtheorem{proposition}[thm]{Proposition}
\theoremstyle{definition}
\newtheorem{remark}[thm]{Remark}
\newtheorem*{defn}{Definition}
\newtheorem{example}[thm]{Example}
\newenvironment{enumeratei}{\begin{enumerate}[\upshape (a)]}
    {\end{enumerate}}

\newenvironment{Enumeratei}{\begin{enumerate}[\upshape (A)]}
    {\end{enumerate}}

\def\irr#1{{\rm Irr}(#1)}
\def\cent#1#2{{\bf C}_{#1}(#2)}
\def\pow#1#2{{\mathcal{P}}_{#1}(#2)}
\def\syl#1#2{{\rm Syl}_#1(#2)}
\def\hall#1#2{{\rm Hall}_#1(#2)}
\def\nor{\trianglelefteq\,}
\def\oh#1#2{{\bf O}_{#1}(#2)}
\def\zent#1{{\bf Z}(#1)}
\def\sbs{\subseteq}
\def\gen#1{\langle#1\rangle}
\def\aut#1{{\rm Aut}(#1)}
\def\out#1{{\rm Out}(#1)}
\def\gv#1{{\rm Van}(#1)}
\def\fit#1{{\bf F}(#1)}
\def\frat#1{{\bf \Phi}(#1)}
\def\gammav#1{{\Gamma}(#1)}
\newcommand{\p}{{\mathbb P}}
\newcommand{\N}{{\mathbb N}}
\newcommand{\F}{{\mathbb F}}
\def\fitd#1{{\bf F}_{2}(#1)}
\def\irr#1{{\rm Irr}(#1)}
\def\dl#1{{\rm dl}(#1)}
\def\h#1{{\rm h}(#1)}
\def\ibr#1#2{{\rm IBr}_#1(#2)}
\def\cs#1{{\rm cs}(#1)}
\def\m#1{{\rm m}(#1)}
\def\n#1{{\rm n}(#1)}
\def\cent#1#2{{\bf C}_{#1}(#2)}
\def\hall#1#2{{\rm Hall}_#1(#2)}
\def\syl#1#2{{\rm Syl}_#1(#2)}
\def\nor{\trianglelefteq\,}
\def\norm#1#2{{\bf N}_{#1}(#2)}
\def\oh#1#2{{\bf O}_{#1}(#2)}
\def\Oh#1#2{{\bf O}^{#1}(#2)}
\def\zent#1{{\bf Z}(#1)}
\def\sbs{\subseteq}
\def\gen#1{\langle#1\rangle}
\def\aut#1{{\rm Aut}(#1)}
\def\gal#1{{\rm Gal}(#1)}
\def\alt#1{{\rm Alt}(#1)}
\def\sym#1{{\rm Sym}(#1)}
\def\out#1{{\rm Out}(#1)}
\def\gv#1{{\rm Van}(#1)}
\def\fit#1{{\bf F}(#1)}
\def\lay#1{{\bf E}(#1)}
\def\fitg#1{{\bf F^*}(#1)}

\def\GF#1{{\rm GF}(#1)}
\def\SL#1{{\rm SL}_{2}(#1)}
\def\PSL#1{{\rm PSL}_{2}(#1)}

\def\gammav#1{{\Gamma}(#1)}
\def\V#1{{\rm V}(#1)}
\def\E#1{{\rm E}(#1)}
\def\b#1{\overline{#1}}

 \def\sl#1#2{{\rm SL}_{#1}(#2)}
 \def\gl#1#2{{\rm GL}_{#1}(#2)}
 \def\cl#1#2{{\rm cl}_{#1}(#2)}
\def\Z{{\mathbb{Z}}}
\def\C{{\Bbb C}}
\def\Q{{\Bbb Q}}
\def\inv{^{-1}}
\def\irr#1{{\rm Irr}(#1)}
\def\irrv#1{{\rm Irr}_{\Bbb R}(#1)}
 \def\irrk#1{{\rm Irr}_{ {\rm rv}, K}(#1)}
 \def\irrc#1{{\rm Irr}_{C}(#1)}
  \def\irrf#1{{\rm Irr}_{\mathfrak{F}'}(#1)}
   \def\ext#1{{\rm Ext}(#1)}
   \def\irrh#1{{\rm Irr}_{H}(#1)}
  \def\re#1{{\rm Re}(#1)}
  \def\csrv#1{{\rm cs}_{\rm rv}(#1)}
   \def\clrv#1{{\rm Cl}_{\rm rv}(#1)}
     \def\clk#1{{\rm Cl}_{{\rm rv}, K}(#1)}
  \def\bip#1{{\rm B}_{p'}(#1)}
  \def\irra#1{{\rm Irr}_A(#1)}
   \def\irrs#1{{\rm Irr}_\sigma(#1)}
   \def\irrp#1{{\rm Irr}_{p'}(#1)}
 \def\cdrv#1{{\rm cd}_{\rm rv}(#1)}
 \def\bip#1{{\rm B}_{p'}(#1)}
\def\cdrv#1{{\rm cd}_{\rm rv}(#1)}
\def\cd#1{{\rm cd}(#1)}
\def\irrat#1{{\rm Irr}_{\rm rat}(#1)}
\def\cdrat#1{{\rm cd}_{\rm rat}(#1)}
\def \c#1{{\cal #1}}
\def\cent#1#2{{\bf C}_{#1}(#2)}
\def\syl#1#2{{\rm Syl}_#1(#2)}
\def\oh#1#2{{\bf O}_{#1}(#2)}
\def\Oh#1#2{{\bf O}^{#1}(#2)}
\def\zent#1{{\bf Z}(#1)}
\def\det#1{{\rm det}(#1)}
\def\ker#1{{\rm ker}(#1)}
\def\norm#1#2{{\bf N}_{#1}(#2)}
\def\alt#1{{\rm Alt}(#1)}
\def\iitem#1{\goodbreak\par\noindent{\bf #1}}
    \def \mod#1{\, {\rm mod} \, #1 \, }
\def\sbs{\subseteq}

\def\Char{\rm Char}
\def\Irr{\rm Irr}
\def\Ext{\rm Ext}
\def\Syl{\rm Syl}

\def\Min#1#2{{\rm Min}_{#1}(#2)}
\def\Mcd#1#2{{\rm Mcd}_{#1}(#2)}
\def\Lcd#1#2{{\rm Lcd}_{#1}(#2)}
\def\lin#1#2{{\rm Lin}_{#1}(#2)}
\def\Lin#1{\rm Lin(#1)}

\setlist[itemize]{font=\color{itemizecolor}}
\colorlet{itemizecolor}{.}
\setlist[enumerate]{font=\color{enumeratecolor}}
\colorlet{enumeratecolor}{.}
\setlist[description]{font=\bfseries\color{descriptioncolor}}
\colorlet{descriptioncolor}{.}

\def\cE{\bar{\rm E}}

\def \nq{\mathfrak{N}_q}

\marginparsep-0.5cm

\renewcommand{\thefootnote}{\fnsymbol{footnote}}
\footnotesep6.5pt

\title[]{On degrees of minimal invariant characters}
\maketitle

\bigskip

\begin{center}
M.J. Felipe, I. Gilabert, L. Sanus
\end{center}

\begin{abstract} It is well known that finite groups with exactly two character degrees have an abelian derived subgroup and, consequently, are solvable. Let $G$ be a finite group and $N$ a normal subgroup of $G$. In this paper, we prove that normal subgroups possessing exactly two degrees of minimal $G$-invariant characters are solvable. Furthermore, it is shown that if these degrees are $\{1, f\}$ for some integer $f$, then either $f$ is a prime power or the commutator subgroup $[N,G]$ is abelian. Whether $[N, G]$ is abelian when $f$ is a prime power remains an open problem. Specifically, we prove that this holds when $f=p$.

\end{abstract}

\bigskip


\bigskip


\bigskip

\thanks{\small The second author is supported by a grant (CIACIF/2023/389 funded the Generalitat Valenciana). The second and third authors are partially supported by the
Spanish Ministerio de Ciencia e Innovaci\'on
(PID2022-137612NB-I00 funded by MCIN/AEI/10.13039/501100011033 and `ERDF A way of making Europe').}

\section{Introduction}

All groups considered in this work are finite. If $G$ is a group, we denote by $\Irr(G)$ the set of irreducible complex characters of $G$, by $\Lin G$ the set of linear characters of $G$, and by $\cd G$ the set of irreducible character degrees of $G$.

\bigskip

If $N$ is a normal subgroup of $G$, then $G$ acts by conjugation on the set $\Irr(N)$. If $\theta \in \Irr(N)$, then $\theta^g \in \Irr(N)$ for all $g \in G$ and
$$\widehat{\theta}=\sum_{i=1}^{t_\theta} \theta^{g_i}=\theta^{g_1}+\cdots+\theta^{g_{t_\theta}}$$
is a \textit{minimal $G$-invariant character} of $N$, where $\{\theta^{g_i} \ |\  1\leq i \leq t_\theta\}$ is the orbit of $\theta$ under the action of $G$ on $\irr N$. That is, the set $\{g_i \ | \ 1 \leq i \leq t_\theta\}$ is a right transversal in $G$ of $I_G(\theta)=\{g \in G \ | \ \theta^g=\theta\}$, the \textit{inertia subgroup} of $\theta$ in $G$. Clifford's theorem (see \cite{Is}, Theorem (6.2)) states that the irreducible characters $\chi$ of $G$ whose restrictions $\chi_N$ to $N$ have $\theta$ as an irreducible constituent satisfy $\chi_N=e \widehat{\theta}$ for suitable integers $e$, known as \textit{ramification numbers}. An irreducible character $\theta$ of $N$ is  \textit{$G$-invariant} if $I_G(\theta)=G$, and in that case $\widehat{\theta}=\theta$.

\bigskip

Next, we fix the notation and terminology that will be used throughout this paper.

\bigskip

If $\theta \in \irr N$, we denote by
$$\irr {G \mid \theta} = \{\chi \in \irr G \mid [\chi_N , \theta] \neq 0\}$$
the set of irreducible characters of $G$ lying over $\theta$, and by $\Irr_G(N)$ the set of irreducible $G$-invariant characters of $N$. We define
$$\text{Lin}_G(N)=\{\theta \in \irr N \mid \widehat{\theta}=\theta \text{ and } \theta(1)=1\},$$
the set of \textit{linear $G$-invariant characters of $N$}.

\smallskip

We denote the set of minimal $G$-invariant characters of $N$ by
$$ \Min G N=\{\widehat{\theta}\mid \theta \in \irr{N}\},$$
the \textit{minimal $G$-invariant character degrees} of $N$ by
$$\Mcd G N = \{\widehat{\theta}(1)\mid \theta \in \irr{N}\},$$
and the set of the {\it leader $G$-character degrees of $N$} (see Definition 3.18 of \cite{AFJP}) by
$$\Lcd G N = \{\theta(1)\widehat{\theta}(1)\mid \theta \in \irr{N}\}.$$

\medskip

Over the last few decades, several researchers have exhibited how the sizes of the $G$-conjugacy classes of $N$ (that is, the unions of the elements in each orbit defined by the natural action of $G$ on the set of conjugacy classes of $N$) influence the structure of $N$ (see, for instance, \cite{BFM}). From a dual perspective (considering the natural action of $G$ on the set of irreducible characters of $N$), one may contemplate whether there exist connections between the set of degrees of minimal $G$-invariant characters of $N$ and its group structure. For instance, well-known results as Thompson's theorem and Burnside's theorem are extended within this context (see Section 2). In Section 3, we prove the main results of the paper, regarding normal subgroups with exactly two minimal $G$-invariant character degrees, and in Section 4 we present some clarifying examples in relation to these results.

\begin{ThmA}
Let $N$ be a normal subgroup of the group $G$. If $\Mcd G N = \{1, f\}$, for some positive integer $f$, then $N$ is solvable.
\end{ThmA}

Corollary (12.6) in \cite{Is} asserts that if $|\textnormal{cd}(G)|=2$, then $G'=[G,G]$ is abelian. We point out that if a normal subgroup $N$ has two minimal $G$-invariant character degrees, it is not necessarily true that $N$ itself has only two irreducible character degrees. Furthermore, in this case, it can be concluded that the derived subgroup $N' = [N, N]$ is abelian, but $[N, G]$ need not be. Within our framework, we conjecture that the following result holds.

\begin{ConjectureB} Let $N$ be a normal subgroup of the group $G$. If $\Mcd G N = \{1, f\}$, for some positive integer $f$, then $[N,G]$ is abelian.
\end{ConjectureB}

We note that the conjecture clearly holds  when the order and the index in $G$ of $N$ are coprime. Indeed, it is not hard to see that either $N$ is abelian or $|\cd N|=2$ and $\Lin N = \text{Lin}_G(N)$, in which case $N'=[N,N]=[N,G]$ is abelian, using Proposition 3.2 of \cite{AFJP} and Corollary (12.6) of \cite{Is}.

The result below shows the conjecture holds when $f$ is not a prime power.

\begin{ThmC}
Let $N$ be a normal subgroup of the group $G$. If $\Mcd{G}{N}=\{1,f\}$, then either $[N,G]$ is abelian or $f$ is a power of $p$, for some prime $p$. In the second case, if $N$ is a $p$-group, then $N$ is a hypercentral subgroup of $G$ and $[N,G] \leq \Phi(G)$, where $\Phi(G)$ denotes the Frattini subgroup of $G$.
\end{ThmC}

Thus, in order to prove Conjecture B, it only remains to consider the case where $f$ is a power of some prime $p$. The particular case where $f=p$ holds.

\begin{ThmD}
Let $N$ be a normal subgroup of the group $G$. If $\Mcd G N = \{1, p\}$, for some prime  $p$, then $[N,G]$ is abelian.
\end{ThmD}

\section{Some properties of minimal $G$-invariant characters}

The minimal $G$-invariant characters of a normal subgroup $N$ of a group $G$ satisfy the following properties.

\begin{lemma} \label{centro} Let $N \unlhd G$ and $\theta \in \irr N$. Then:

\begin{itemize}

\item[(i)] if $\chi \in \irr{G \mid \theta}$, then $\ker{\chi} \cap N=\ker{\widehat{\theta}} = \bigcap_{g \in G} \ker{\theta^g}$;

\item[(ii)] if $\chi \in \irr{G \mid \theta}$, then $\zent \chi \cap N = \zent {\widehat{\theta}} \subseteq \bigcap_{g \in G} \zent {\theta}^g$; 

\item[(iii)] $\bigcap_{\widehat{\theta} \in \Min G N} \zent {\widehat{\theta}} = \zent G \cap N \subseteq \zent N$. If $\widehat{\theta}$ is faithful, then $\zent {\widehat{\theta}}= \zent G \cap N$;

\item[(iv)] the determinantal orders $o(\widehat{\theta})=o(\theta)$ are equal.

\end{itemize}

\end{lemma}

\begin{proof} (i) Let $\theta \in \Irr(N)$ and let $\widehat{\theta}$ be the minimal $G$-invariant character of $N$ associated to $\theta$. If $\chi \in \irr{G \mid \theta}$, by Clifford's theorem, we obtain that $\ker{\chi} \cap N = \ker{{\widehat{\theta}}}$. The other equality follows trivially since $\widehat{\theta}$ is a character of $N$ and $\{ \theta^g \mid g \in G \}$ is the set of its irreducible constituents.\\

(ii) By Clifford's theorem, we obtain trivially that $\zent \chi \cap N = \zent {\widehat{\theta}}$. Moreover, we have $\widehat{\theta}=\theta_1+\cdots+\theta_{t_\theta}$ where $\theta_1=\theta$ and $\{\theta_1,\hdots,\theta_{t_\theta}\}$ are the distinct $G$-conjugates of $\theta$. Let $x \in \zent {\widehat{\theta}}$, and let $\mathcal{R}$ be a representation of $N$ affording $\widehat{\theta}$. We have $\mathcal{R}(x)=\varepsilon I_{\widehat{\theta}(1)}$ for some $\varepsilon \in \mathbb{C}$, where $\mathcal{R}(x)$ is similar to the block diagonal matrix with blocks $\mathcal{R}_1(x),\hdots,\mathcal{R}_{t_\theta}(x)$, where $\mathcal{R}_i$ is a representation of $N$ affording $\theta_i$ for all $i \in \{1,\hdots,t_\theta\}$. Since only $\varepsilon I_{\widehat{\theta}(1)}$ is similar to itself, we deduce that $\mathcal{R}_i(x)=\varepsilon I_{\theta(1)}$, hence $x \in \zent {\theta_i}$ for all $i$. Therefore, $\zent{\widehat{\theta}} \subseteq \bigcap_{i=1}^{t_\theta} \zent{\theta_i}= \bigcap_{g \in G} \zent {\theta}^g$.
\\

(iii) Trivially, $\zent G \cap N \subseteq \zent {\widehat{\theta}}$, for every $\widehat{\theta} \in \text{Min}_G(N)$. If $g \in \bigcap_{\widehat{\theta} \in \text{Min}_G(N)} \zent {\widehat{\theta}}$, then $g \,\ker {\widehat{\theta}} \in \zent { G/\ker {\widehat{\theta}}}$, and thus 
$$[g,x] \in \ker {\widehat{\theta}} \subseteq \ker \theta,$$
for every $x \in G$ and $\theta \in \irr N$. Then $[g,x]=1$ and $g \in \zent G \cap N$.

Suppose that $\widehat{\theta}$ is a faithful character of $N$. Consider $\chi \in \Irr(G | \theta)$. Since $\chi_N=e \widehat{\theta}$ for some integer $e$, we have that $\zent \chi \cap N = \zent{\widehat{\theta}}$. In particular, $\zent {\widehat{\theta}} \subseteq \zent \chi$. Using that $\zent \chi / \ker \chi = \zent{G/\ker \chi}$, we have that if $x \in \zent {\widehat{\theta}}$ then $x \,\ker \chi \in \zent{G/\ker \chi}$. Therefore, $[x,G] \in \ker \chi$. But $[x,G] \in [N,G] \subseteq N$, and so $[x,G] \in \ker \chi \cap N = \ker {\widehat{\theta}} = 1.$ That is, $x \in \zent G$. It follows that $\zent{\widehat{\theta}} = \zent G \cap N$. \\

(iv)  If $g \in G$, then
$$\det {\theta^g} = \det{\theta}^g,$$
and so the determinantal orders $o(\theta^g)$ and $o(\theta)$ are equal.
Since
$$\det {\widehat{\theta}} = \prod_{g \in G} \det{\theta^{g}}$$
where all $\det{\theta^{g}}$ have the same order, we conclude $o(\widehat{\theta})=o(\theta)$.
\end{proof}
 
\begin{remark} Under the conditions of Lemma \ref{centro}(ii) we do not necessarily have $\zent{\widehat{\theta}}= \bigcap_{g \in G} \zent{\theta}^g$. See Example \ref{ej 1} in Section 4 for that matter.
 \end{remark}

In the following result, we prove a light generalization of Burnside's theorem (see Theorem (3.8) of \cite{Is}) for minimal $G$-invariant characters.

\begin{proposition} \label{burnside}
    Let $N$ be a normal subgroup of $G$ and let $\widehat{\theta} \in \Min G N$. Let $K=g^G$ be the $G$-conjugacy class of an element $g \in N$. If $(\widehat{\theta}(1), |K|) = 1$, then either $g \in \zent{\widehat{\theta}}$ or $\widehat{\theta}(g) = 0$.
\end{proposition}
\begin{proof}
    
Let $\chi \in \irr{G \mid \theta}$. By Clifford's theorem, we have $\chi(n) = e \,\widehat{\theta}(n)$ for all $n \in N$,
where $e$ is a nonnegative integer. Then it follows that $$\frac{\widehat{\theta}(g)}{\widehat{\theta}(1)} = \frac{\chi(g)}{\chi(1)}.$$
 Now, by Bézout's identity, there exist integers $a, b\in \mathbb{Z}$ such that $a \, \widehat{\theta}(1) + b \,|K| = 1$. From here, we have
$$\frac{\chi(g) \, b|K|}{\chi(1)} =  \frac{\widehat{\theta}(g) \,b|K|}{\widehat{\theta}(1)} = \frac{\widehat{\theta}(g) (1 - a \,\widehat{\theta}(1))}{\widehat{\theta}(1)} = \frac{\widehat{\theta}(g)}{\widehat{\theta}(1)} - a\,\widehat{\theta}(g)$$
is an algebraic integer. Therefore, since $a\,\widehat{\theta}(g)$ is an algebraic integer, $\widehat{\theta}(g)/\widehat{\theta}(1) = \chi(g)/\chi(1)$ is also an algebraic integer.  

Let $k$ be the order of the element $g \in N$. Since
$\chi(g) = \varepsilon_1 + \cdots + \varepsilon_f,
$
where each $\varepsilon_i$ (for $i = 1, \ldots, f$) is a root of the polynomial $x^k - 1$, it follows that $\chi(g)/\chi(1)$ lies in the splitting field $E$ of the polynomial $x^k - 1$. Let $\mathcal{G} = \mathrm{Gal}(E \mid \mathbb{Q})$. For $\sigma \in \mathcal{G}$, we have
$\chi(g)^\sigma = \varepsilon_1^\sigma + \cdots + \varepsilon_f^\sigma$.

Assume that $g \not\in \zent \chi$. Then $|\chi(g)| < \chi(1)$ and also $|\chi(g)^\sigma| < \chi(1)$ for all $\sigma \in \mathcal{G}$. Hence
$$
\left| \frac{\chi(g)^\sigma}{\chi(1)} \right| < 1.
$$
Define
$$
\beta = \prod_{\sigma \in \mathcal{G}} \frac{\chi(g)^\sigma}{\chi(1)}.
$$
Then $\beta$ is an algebraic integer such that $\beta^\sigma = \beta$ for all $\sigma \in \mathcal{G}$, and thus $\beta$ is an integer. However, we have $|\beta| < 1$, so $\beta = 0$. Therefore, $\chi(g) = 0$ and $\widehat{\theta}(g) = 0$ as desired.

\end{proof} 

\begin{remark}\label{remark burnside} Regarding the above proposition, note that the hypothesis $(\widehat{\theta}(1),|K|)=1$ does not imply that $(\chi(1),|K|)=1$, for $\chi \in \irr{G \mid \theta}$. On the other hand, if $(\widehat{\theta}(1),|K|)=1$, then $(\theta(1),|g^N|)=1$, where $g^N$ is the $N$-conjugacy class of $g$, and it follows by Burnside's theorem that $g \in \zent \theta$ or $\theta(g)=0$. However, this still does not imply that $g \in \zent{\widehat{\theta}}$ or $\widehat{\theta}(g)=0$. This is seen in Example \ref{ej 1}.
\end{remark}

\begin{corollary}

Let $N$ be a minimal normal subgroup of $G$. Suppose that $K$ is the $G$-conjugacy class of a non-central element $g \in N$ such that $|K|$ is a power of a prime $p$. Then $p$ divides $|G:N|$.
\end{corollary}

\begin{proof}
Let $\theta \in \irr N \setminus \{1_N\}$. Since $\ker{\widehat{\theta}}$ is a normal subgroup of $G$ contained in $N$, there are only two possibilities: either $\ker{\widehat{\theta}} = 1$ or $\ker{\widehat{\theta}} = N$. The latter cannot occur because $\theta$ is not the trivial character of $N$. Therefore, $\ker{\widehat{\theta}} = 1$, and by Lemma \ref{centro} we have
$$
\zent{\widehat{\theta}} = \zent{G} \cap N.
$$

Now observe that if $p$ does not divide $\widehat{\theta}(1)$,  by the previous proposition, it follows that either $g \in \zent{\widehat{\theta}}$ or $\widehat{\theta}(g) = 0$. The first option is impossible, since $g \notin \zent{G}$ by hypothesis. Hence
$$
\widehat{\theta}(g) = 0.
$$
Let
$$
\Delta = \{ \chi \in \irr G \mid N \subseteq \ker\chi \}.
$$
Note that for every $\chi \in \irr G \setminus \Delta$, we have either $\chi(g) = 0$ or that $p$ divides $\chi(1)$.

By the second orthogonality relation, we obtain
$$
0 = \sum_{\chi \in \irr G} \chi(1)\chi(g)
  = \sum_{\chi \in \Delta} \chi(1)\chi(g)
  + \sum_{\chi \in \irr G  \setminus \Delta} \chi(1)\chi(g)= $$
  $$\sum_{\chi \in \Delta} \chi(1)^2
  + p \sum_{\chi \in \irr G  \setminus \Delta} \dfrac{\chi(1)}{p}\chi(g).
$$

Since the first sum equals $|G:N|$ and the second sum (which is denoted by $a$) is an algebraic integer, we may write
$$
0 = |G:N| + p a,
$$
and thus $a=-|G:N|/p \in \mathbb{Q}$ is an integer. Therefore, $p$ divides $|G:N|$, as desired. 
\end{proof}

Next, we list some preliminaries that will be needed for the proofs of the main results.

\begin{lemma}[\cite{AFJP}, Proposition 3.2] \label{[G,N] ker(theta) LinGN} Let $N$ be a normal subgroup of a group $G$ and let $\theta \in \Irr(N)$. Then $[N,G] \subseteq \ker \theta$ if and only if $\theta \in \lin G N$. In particular, $|\,N:[N,G]\,|=|\,\lin G N\,|$.
\end{lemma}

\begin{lemma} \label{induccion} Let $N \unlhd G$.
\begin{itemize}
    \item[(i)] If $\Mcd G N = \{1\}$, then $N$ is central in $G$.
    \item[(ii)] Let $K \unlhd G$ such that $K \leq N$. Then
    $$\Mcd {G/K} {N/K} \subseteq \Mcd G N.$$
    As a consequence,
    $$\Lcd {G/K} {N/K} \subseteq \Lcd G N.$$
    \item[(iii)] Let $M \cdot \unlhd \ G$. Then
    $$\Mcd {G/M} {NM/M} \subseteq \Mcd G N.$$
    As a consequence,
    $$\Lcd {G/M} {NM/M} \subseteq \Lcd G N.$$
\end{itemize}
\end{lemma}

\begin{proof} (i) Since $\lin G N=\Irr(N)$, by the above lemma, we have
$$[N,G] \subseteq \bigcap_{\theta \in \Irr(N)}\ker \theta = 1,$$
thus $N \subseteq \zent G$.

(ii) As usual, we denote $\overline{G}=G/K$ and $\overline{N}=N/K$. Let $\widehat{\theta} \in \Min {\overline{G}} {\overline{N}}$ with $\theta \in \irr{\overline{N}}$. By inflation, one may consider $\theta$ as a character in $\irr N$ such that $K \subseteq \ker \theta \subseteq N \subseteq I_G(\theta) \subseteq G$. We use the same notation $\theta$ for both cases. 

Now, we prove that $I_G(\theta)/K = I_{\overline{G}}(\theta)$. We have that $\theta(n)=\theta(\overline{n})$, for every $n \in N$. Trivially, $I_G(\theta)/K \subseteq I_{\overline{G}}(\theta)$. Reciprocally, if $\overline{x} \in I_{\overline{G}}(\theta)$ and $\overline{n} \in \overline{N}$, then
$$\theta(n)=\theta(\overline{n})=\theta^{\overline{x}}(\overline{n})=\theta(\overline{xnx^{-1}})=\theta^x(n),$$
for every $n \in N$, $x \in I_G(\theta)$, and $I_G(\theta)/K = I_{\overline{G}}(\theta)$. Thus
$$\widehat{\theta}(\overline{1})=\left|\overline{G}:I_{\overline{G}}(\theta)\right|\theta(\overline{1})=|G: I_G(\theta)|\,\theta(1) \in \Mcd G N,$$
and similarly
$$\widehat{\theta}(\overline{1}) \, \theta(\overline{1})=\left|\overline{G}:I_{\overline{G}}(\theta)\right|\theta^2(\overline{1})=|G: I_G(\theta)|\,\theta^2(1) \in \Lcd G N.$$

(iii)  Using (ii), we may assume that $M \not \subseteq N$. By minimality of $M$, we have that $M \cap N = 1$ and $MN= M \times N$. Consider $\theta \in \Irr(N)$. Then $I_G(\theta) = I_G(\theta \times 1_M)$. By the same reasoning as above, by identifying $\theta \times 1_M$ with an irreducible character of $NM/M$, we have that $I_{G/M}(\theta \times 1_M)=I_G (\theta \times 1_M)M/M = I_G(\theta \times 1_M)/M$. Thus
\begin{align*}
    \Mcd G N & = \{ |G:I_G(\theta \times 1_M)|\,(\theta \times 1_M)(1),  \ \theta \in \Irr(N) \} \\
    & = \{ |G/M:I_{G/M}(\theta \times 1_M)|\,(\theta \times 1_M)(1) , \ \theta \times 1_M \in \Irr(NM/M) \}
    \\ & = \Mcd{G/M} {NM/M}.
 \end{align*}

 Trivially, since $\theta(1)=(\theta\times 1_M)(1)$, we also have that
$$\Lcd G N = \Lcd {G/M}{NM/M}.$$
\end{proof}

The well-known Thompson's theorem is generalized to minimal $G$-invariant characters of $N$ as follows.

\begin{theorem}[\cite{AFJP}, Theorem 3.10] \label{Thompson}
Let $N$ be a normal subgroup of the group $G$ and let $p$ be a prime integer. Suppose $p$ divides $\widehat{\theta}(1)$ for every $\widehat{\theta}\in \Min G N \setminus \lin G N$. Then $N$ has a normal $p$-complement.
\end{theorem}

\bigskip

For a normal subgroup $N$ of $G$, we may construct the following descending series:
$$N=\Gamma^1_G(N) \trianglerighteq \Gamma^2_G(N)=[\Gamma^1_G(N),G] \trianglerighteq \Gamma^3_G(N)=[\Gamma^2_G(N),G] \trianglerighteq \cdots \ ,$$
where $\Gamma^{i+1}_G(N)=[\Gamma^i_G(N),G]$ for any integer $i\geq 1$. This was defined in \cite{AFJP} as the \textit{lower central $G$-series of $N$}. As stated in \cite{AFJP}, $N$ is a hypercentral subgroup of $G$ (that is, is contained in the hypercenter of $G$) if and only if there exists some integer $r \geq 1$ such that $\Gamma^{r+1}_G(N)=1$. In that case, we say that the least such integer $r$ is the \textit{hypercentral $G$-length of $N$}.

\medskip

The following concept was also introduced in \cite{AFJP}: a normal subgroup $N$ of a group $G$ is a \textit{$G$-invariant $nMI$-subgroup} if, for every $\theta \in \irr{N}$, there exist $H_{\theta} \unlhd G$ and $\lambda_{\theta} \in \text{Lin}_G(H_{\theta})$ such that $\irr{G \mid \theta}=\irr{G \mid \lambda_{\theta} }$. We say that $(H_{\theta}, \lambda_{\theta})$ is a \textit{linear $G$-invariant character pair with respect to $\theta$}.

A variation of Taketa's theorem for $G$-invariant $nMI$-subgroups was published in \cite{AFJP}. We note that the superscript of $\Gamma_G^{i}(N)$ in the initial statement has been corrected to $\Gamma_G^{i+1}(N)$, according to the later version available on Arxiv (see \cite{AFJP2}).
 
 \begin{theorem}[\cite{AFJP}, Theorem 3.19]\label{nMI}
 Let $N$ be a $G$-invariant $nMI$-subgroup of a group $G$ and let $1=f_1 < f_2 < \cdots < f_s$  the distinct elements in $\textnormal{Lcd}_G(N)$. Then
$$\Gamma^{i+1}_G(N) \subseteq \ker {\widehat{\theta_{i}}}$$
with $\theta_i(1)\widehat{\theta_i}(1) = f_i$ and $\theta_i \in \irr{N}$, for $i \in \{1, \ldots, s\}$. In particular, $N$ is a hypercentral subgroup of $G$ and $l_G(N) \leq s=|\Lcd{G}{N}|$.
\end{theorem}

\section{Main results} \label{sección main results}

Recall that if $G$ is a group with two irreducible character degrees, then the derived subgroup of $G$ is abelian by Corollary (12.6) in \cite{Is}, and thus $G$ is solvable. We begin this section by proving analogously that if $N$ is a normal subgroup of a group $G$ having two minimal $G$-invariant character degrees, then $N$ is solvable. As we mentioned in the introduction, this is not a corollary of the first result, since $N$ could have three or more irreducible character degrees, but only two minimal $G$-invariant character degrees. The group $G=\texttt{SmallGroup}(128,2264)$ on GAP yields an example of this, as it is possible to choose $N \unlhd G$ of size $64$ with $\Mcd G N = \{1,4\}$ but $\cd N = \{1,2,4\}$.

We shall make use of the following key result, which depends on the classification of finite simple groups.

\begin{lemma}[\cite{BCLP}, Lemma 5] \label{lemma1} Let $N$ be a normal subgroup of $G$ so that $N = S_1 \times \cdots \times S_t$, where $S_i \cong S$, a nonabelian simple group. Let $A$ be the automorphism group of $S$. If $\sigma \in \Irr(S)$ extends to $A$, then $\sigma \times \cdots \times \sigma$ extends to $G$.
\end{lemma}

\medskip

\begin{proof}[Proof of Theorem A] We first prove that we can assume there exists a unique minimal normal subgroup $M$ of $G$. Suppose instead that there exist $M_1 \neq M_2 \cdot \unlhd \ G$. By Lemma \ref{induccion} (iii), we have $$|\Mcd {G/M_i} {NM_i /M_i}|\leq 2$$
for $i \in \{1,2\}$. Then, by induction, both $NM_1/M_1$ and $NM_2/M_2$ are solvable.

Now, we may assume that $M_1 \subseteq N$ and $M_2 \subseteq N$. But $M_1 \cap M_2 =1$ by minimality of $M_1$ and $M_2$, and thus we may see $N$ as a subgroup of $N/M_1 \times N/M_2$. Since it is a direct product of solvable groups, $N$ is solvable.

Therefore, we consider $M$ the unique minimal normal subgroup of $G$. By uniqueness, we have $M \subseteq N$. Suppose $M$ is not abelian. We write
$$M= S_1 \times \cdots \times S_t,$$
where $S_1,\hdots,S_t$ are isomorphic copies of a nonabelian simple group $S$.

Assume $S$ is a simple group of Lie type with defining characteristic $p$, for some prime $p$. Let $\varphi$ be the Steinberg character of $S$. Then we know that $\varphi(1)$ is a power of $p$, say $p^a$, and, by Lemma \ref{lemma1}, that $\varphi \times \cdots \times \varphi \in \Irr(M)$ extends to $\chi \in \Irr(G)$. Then $\chi_N$ is an irreducible $G$-invariant character of degree $p^{ta}$ and $\Mcd G N=\{1,p^{ta}\}$. By Theorem \ref{Thompson}, $N$ has a normal $p$-complement. Since every character of the normal $p$-complement must be linear, it is abelian, which contradicts $M$ not being abelian.

Suppose now that $S$ is either isomorphic to the alternating group $A_n$ for $n \geq 7$, the Tits group, or a sporadic simple group. It is then possible to find two nonlinear characters $\sigma_1, \sigma_2 \in \Irr(S)$ such that $(\sigma_1(1),\sigma_2(1))=1$ and that both $\tau_1=\sigma_1 \times \cdots \times \sigma_1 \in \Irr(M)$ and $\tau_2 = \sigma_2 \times \cdots \times \sigma_2 \in \Irr(M)$ extend to $G$ (we use Theorems 3 and 4 of \cite{BCLP} and Lemma \ref{lemma1}). Let $\chi_1, \chi_2 \in \Irr(G)$ restricting to $\tau_1$ and $\tau_2$ respectively. Then $\theta_1 = (\chi_1)_N$ and $\theta_2 = (\chi_2)_N$ are irreducible $G$-invariant characters of $N$ with relatively prime degrees $\sigma_1(1)^t$ and $\sigma_2(1)^t$. In particular, $\theta_1(1) \neq \theta_2(1)$, which contradicts $\Mcd G N$ only having two elements.

Therefore, $M$ must be abelian. By induction, $N/M$ is solvable, and thus $N$ is solvable and Theorem A is proven.
\end{proof}

Although a normal subgroup $N$ of $G$ with only two minimal $G$-invariant character degrees is solvable by Theorem A, it is not necessarily nilpotent, and as a consequence not hypercentral in $G$ either (see Example \ref{ej 2}). However, if $N$ is a $p$-group, it is hypercentral in $G$ under the additional hypothesis that all minimal $G$-invariant character degrees of $N$ are powers of $p$.

\begin{proposition}\label{p grupo hipercentral}
Let $N$ be a normal subgroup of the group $G$. If $N$ is a $p$-group for some prime $p$ and $\widehat{\theta}(1)$ is a $p$-power number for all $\widehat{\theta} \in \Min G N$, then $N$ is a hypercentral subgroup of $G$. In particular, $[N, G]\subseteq {\bf \Phi}(G)$.
\end{proposition}

\begin{proof}
Note first that, since $N$ is a $p$-group, $1<\zent N$. Let $L$ be a minimal normal subgroup of $G$ contained in $\zent N$. We claim $L \subseteq \zent G$. Suppose otherwise that $1 \neq [L,G]$. Then by minimality we have $L=[L,G]$ and thus $\text{Lin}_G(L)=\{1_L\}$. Let $\lambda \in \irr L \backslash \{1_L\}$ and consider $\theta \in \irr {N \mid \lambda}$. Since $L \subseteq \zent N$, we have that $\theta_L=\theta(1) \lambda$. Hence $I_G(\theta) \subseteq I_G(\lambda) \subsetneq G$. By hypothesis, $|G:I_G(\theta)|$ is a power of $p$. Then $\widehat{\lambda}(1)=|G:I_G(\lambda)|>1$ is again a $p$-power number. It follows that
$$|L|=1+\sum_{\substack{\lambda \in \irr L \\ \lambda \neq 1_L}}\lambda(1)^2 = 1 + \sum_{\substack{\widehat{\lambda} \in \Min G L \\ \lambda \neq 1_L}} \widehat{\lambda}(1) \lambda(1) = 1+ p \sum_{\substack{\widehat{\lambda} \in \Min G L \\ \lambda \neq 1_L}} \frac{\widehat{\lambda(1)}}{p}\lambda(1),$$
which contradicts the fact that $|L|$ is a power of $p$. Thus, $L \subseteq \zent G$.

By Lemma \ref{induccion}(ii), we have that $\Mcd {G/L} {N/L} \subseteq \Mcd G N$. The factor group $N/L$ is again a $p$-group so, by induction, we may assume that $N/L$ is hypercentral in $G/L$. Then there exists some nonnegative integer $n$ such that
$$[N/L,\underbrace{G/L,\hdots,G/L}_{n \text{-times}}\ ] = 1.$$
Then
$$[N,\underbrace{G,\hdots,G}_{n \text{-times}}\ ] \subseteq L \subseteq \zent G.$$
As a consequence,
$$[N,\underbrace{G,\hdots,G}_{(n+1) \text{-times}}\ ]=1$$
and $N$ is hypercentral in $G$, as desired. Finally, it follows from Corollary 3.9 in \cite{AFJP} that $[N,G] \subseteq \Phi(G)$.
 \end{proof}

 The next corollary follows immediately from Proposition \ref{p grupo hipercentral} when $|\Mcd G N|=2$.

 \begin{corollary} \label{corolario p grupo hipercentral} Let $N$ be a normal $p$-subgroup of a group $G$, for some prime $p$, and $\Mcd G N = \{1,f\}$ with $f=p^a$ for some integer $a$, then $N$ is a hypercentral subgroup of $G$. In particular, $[N,G] \subseteq \Phi(G)$.
 \end{corollary}

 \begin{remark}
Both the hypotheses that $N$ is a $p$-group and that $f$ is a $p$-power are necessary, as shown in Example \ref{ej 2}.
\end{remark}

In order to prove Theorem C, we will need an extension of Lemma (12.3) in \cite{Is}. Note that the conditions fixed by the hypotheses on $[N,G]$ in the following lemma do not imply that $[N,G]$ is a minimal normal subgroup of $N$.

\begin{proposition}\label{12.3 mcd} Let $G$ be a finite group and let $N$ be a solvable normal subgroup of $G$ such that $[N,G]$ is the unique minimal normal subgroup of $G$ contained in $N$. Then one of the following situations occurs:
\begin{itemize}
    \item[(a)] $N$ is an abelian $p$-group and $N/(\zent G \cap N)$ is an elementary abelian $p$-group for some prime $p$;
    \item[(b)] $N$ is not abelian and $\cd N=\{1,m\}$ for some integer $m$. Moreover, one of the following sub-cases occurs.
    \begin{itemize}
        \item[(b.1)] $N$ is a $p$-group and $N/(\zent G \cap N)$ is elementary abelian. Also, $m^2= |N:\zent N|$ (so $m$ is a power of $p$), $\Mcd G N = \{1,m \ |\zent N : \zent G \cap N|\}$, and $\Lcd G N= \{1,m^2 \ |\zent N : \zent G \cap N|\}=\{1,|N: \zent G \cap N|\}$. 
        \item[(b.2)] $N$ is a Frobenius group with an abelian Frobenius complement. Also, $[N, G]$ is the Frobenius kernel and is an elementary abelian $p$-group for a prime $p$. We have that $m=|N:[N,G]|$. Moreover, if $\Mcd G N=\{1,f\}$ for an integer $f$, then $p$ does not divide $f$.
    \end{itemize}
\end{itemize}
\end{proposition}

\begin{proof} By hypothesis, we may assume that $[N, G] \neq 1$. Note that since $[N, G] \cdot \unlhd \ G$ and $[N, G] \subseteq N$ where $N$ is solvable, we have that $[N, G]$ is an elementary abelian $p$-group for some prime $p$. We distinguish two cases depending on whether  $\zent G \cap N$ is trivial or not.\\

{\it Case 1}: Suppose that $\zent G \cap N \neq 1$. Since $[N, G]$ is the unique minimal normal subgroup of $G$ contained in $N$, it follows that $[N, G] \subseteq \zent G \cap N$ and $\zent G \cap N$ is a direct product of cyclic groups, which are all normal in $G$. By hypothesis, this implies that $\zent G \cap N$ is a cyclic $p$-group and $[N, G]$ is a cyclic group of $p$ elements. Moreover, $[N, G] < N$. Otherwise, $[N, G]=N$ and thus $N \subseteq \zent G$, which would imply $[N, G]=1$ and contradict our hypotheses.

Let $x \in G$ and $y \in N$. We have $[x,y] \in [N, G] \subseteq \zent G$ and so it can easily be proven that $[x,y^p]=[x,y]^p$.
As $[N, G] \cong C_p$, it follows that $[x,y^p]=1$ for all $x \in G, y \in N$, \textit{i.e.} $y^p \in \zent G$ for all $y \in N$. Moreover, $N'=[N,N] \subseteq [N, G] \subseteq \zent G \cap N$. Hence $N/(\zent G \cap N)$ is an elementary abelian $p$-group.
Since $\zent G \cap N$ is a $p$-group, we deduce that $N$ is also a $p$-group.

If $N'=1$, then in particular $N$ is an abelian $p$-group and statement (a) follows. We may assume that $N$ is not abelian. Then $N'\leq[N, G]$ is a normal subgroup of $G$ contained in $N$, and thus $N'=[N, G]$ by minimality of the latter. If $\theta \in \Irr(N)$ is linear, then $[N, G]=N' \subseteq \ker \theta$. Then by Lemma \ref{[G,N] ker(theta) LinGN}, we have that $\theta=\widehat{\theta}\in \text{Lin}_G(N)$.

Now we consider a nonlinear character $\theta \in \Irr(N)$. Denote by $\widehat{\theta}$ the minimal $G$-invariant character associated to $\theta$. Suppose $1 \neq \ker {\widehat{\theta}}$. We have that 
$$\ker {\widehat{\theta}}=\bigcap_{g \in G} \ker{\theta^{g}}= \bigcap_{g \in G} (\ker \theta)^{g}.$$
Thus, $\ker {\widehat{\theta}} = \text{Core}_G(\ker \theta) \unlhd G$. Since $\ker {\widehat{\theta}} \subseteq N$, we must have that $[N, G] \subseteq \ker {\widehat{\theta}}$. Then $N' = [N,G] \subseteq \ker {\widehat{\theta}} \subseteq \ker \theta$, hence a contradiction with the non-linearity of $\theta$. So $\ker {\widehat{\theta}}=1$.

By Problem 2.13 in \cite{Is}, we have that $\theta(1)=\sqrt{|N:\zent N|}$, for all nonlinear $\theta \in \Irr(N)$. Hence $m=\theta(1)=\sqrt{|N:\zent N|}$ is a constant value on $\Irr(N) \backslash \Lin N$. Thus $\cd N=\{1,m\}$ with $m^2=|N: \zent N|$.

We claim that $\widehat{\theta}$ vanishes on $N \backslash (\zent{G} \cap N)$ for all nonlinear $\theta \in \Irr(N)$. Let $\chi \in \irr G$ such that $\chi_N=e \widehat{\theta}$ and let $\mathcal{R}$  be a representation of $G$ affording $\chi$. Let $g\in N \backslash (\zent{G} \cap N)$. There exists $x \in G$ such that $[g,x]\neq 1$ and so $[g,x]\not \in \ker \chi$. Write $z=[g,x]$. Since $z \in [N, G] \subseteq \zent G$, there exists some $\varepsilon \in \mathbb{C}$ such that $\mathcal{R}(z)=\varepsilon I$. Furthermore, since $z \not \in \ker \chi$, we must have $\mathcal{R}(z) \neq \mathcal{R}(1)$ and $\varepsilon \neq 1$. We have that
$$\mathcal{R}(gz)=\mathcal{R}(g)\mathcal{R}(z)=\varepsilon \mathcal{R}(g),$$
and also
$$\mathcal{R}(gz)=\mathcal{R}(x^{-1}gx).$$
By taking traces, we see that $\chi(g)=\varepsilon\chi(g)$ and thus $\varepsilon \widehat{\theta}(g)=\chi(g)=0$ since $\varepsilon \neq 1$. That
is, $\widehat{\theta}(g)=0$ for all $g\in N \backslash (\zent{G} \cap N)$.

Now, by Lemma (2.29) in \cite{Is},
$$[\widehat{\theta}_{\zent{G} \cap N},\widehat{\theta}_{\zent{G} \cap N}]=|N:\zent{G} \cap N|\,[\widehat{\theta},\widehat{\theta}]=|N:\zent{G} \cap N|\, t_\theta.$$
But $\theta_{\zent{G} \cap N}=\theta(1)\lambda$ for some $G$-invariant $\lambda \in \Lin{\zent{G} \cap N}$, hence $\widehat{\theta}_{\zent{G} \cap N}=\theta(1)(\lambda^{g_1}+\cdots+\lambda^{g_{t_\theta}})=t_\theta \theta(1) \lambda$. Then
$$[\widehat{\theta}_{\zent{G} \cap N},\widehat{\theta}_{\zent{G} \cap N}]=t_\theta^2 \, \theta(1)^2 \,[\lambda,\lambda]=t_\theta^2\, \theta(1)^2 .$$
Therefore $|N:\zent G \cap N|=t_\theta \theta(1)^2$. Using that $\theta(1)^2=|N:\zent N|$, we deduce that $t_\theta=|\zent N:\zent G \cap N|.$

It follows then that
$$\Mcd{G}{N}=\{1,m t_\theta\}=\{1,\sqrt{|N:\zent N|} \, |\zent N:\zent G \cap N|\}$$
and
$$\Lcd{G}{N}=\{1,m^2 t_\theta\}=\{1,|N : \zent G \cap N|\}.$$
This is case (b.1).

\bigskip

{\it Case 2:} Suppose that $\zent G \cap N = 1$. As previously, we know that $[N, G]$ is an elementary abelian $p$-group for some prime $p$.

Assume $\zent N>1$. Since $\zent N$ is a nontrivial normal subgroup of $G$ contained in $N$, we have $[N, G] \subseteq \zent N$. If $x \in G$ and $n \in N$, then $[x,n] \in [N,G] \subseteq \zent N$ and we have
$$[x,n^p]=[x,n]^p=1,$$
using that $[N,G]$ is an elementary $p$-group. Therefore, $n^p \in \zent G \cap N = 1$ for all $n \in N$ and $N$ is an elementary $p$-group. Moreover, since $[N, G] \subseteq \zent N$, we have that $N$ centralizes $[N, G]$. Then, by the three subgroups lemma (see Lemma 4.9 in \cite{Is08}), we have that $[N,N,G]=1$, that is, $N'$ is central in $G$. Then $N' \subseteq \zent G \cap N=1$, \textit{i.e.} $N'=1$, and we are in case (a).

\smallskip

Assume otherwise that $\zent N=1$. Then $N$ is not abelian and $N' \neq 1$. As a consequence, $N'=[N,G]$. Moreover, $N$ cannot be a $p$-group. Consider then a prime $q\neq p$ such that $q \mid |N|$, and let $Q \in \Syl_q(N)$. 

Since $N'Q=[N, G]Q \unlhd N$, by Frattini's argument, we have 
$$N=[N, G]Q {\bf N}_N(Q)=[N, G]{\bf N}_N(Q)$$
and $G=[N, G] {\bf N}_G(Q)$. Moreover, $[N, G] \cap {\bf N}_G(Q) \unlhd {\bf N}_G(Q)$ because $[N,G] \unlhd G$ and $[N, G] \cap {\bf N}_G(Q) \unlhd [N, G]$  since $[N, G]$ is abelian, thus $[N, G] \cap {\bf N}_G(Q) \unlhd G$. If $1 \neq [N, G] \cap {\bf N}_N(Q)$, then by minimality of $[N, G]$ we must have $[N, G]=[N, G] \cap {\bf N}_N(Q)$, that is, $[N, G] \subseteq {\bf N}_N(Q)$. Then $G={\bf N}_G(Q)$. But then $Q$ is a normal subgroup of $G$ contained in $N$, hence $[N, G] \subseteq Q$, which contradicts the fact that $p \neq q$. Thus, $[N, G] \cap {\bf N}_N(Q)=1$ and
$${\bf N}_N(Q) \cong N/[N, G] = N/N'.$$
Therefore ${\bf N}_N(Q)$ is abelian. Consider a nonidentity element $x \in {\bf N}_N(Q)$. Since $N=[N, G] {\bf N}_N(Q)$ and both factors are abelian groups, we have ${\bf C}_{[N, G]}(x) \unlhd N$.

Let us verify that ${\bf C}_{[N, G]}(x)=1$. Suppose there exists a nonidentity $a \in {\bf C}_{[N, G]}(x)={\bf C}_{[N, G]}(x^{-1})$. If $g \in G$, then we may write $x^g=nx$ for some $n \in [N, G]$. But, $[N, G]$ being abelian, we have
$$[a,x^g]=[a,nx]=[a,x][a,n]^x=[a,n]^x=1.$$
Then for all $g \in G$ we have $a \in {\bf C}_{[N, G]}(x^g)$, and thus $$a \in L=\bigcap_{g \in G} ({\bf C}_{[N, G]}(x))^g.$$ Since $L$ is a nontrivial normal subgroup of $G$ contained in $[N, G]$, we deduce by minimality that $L=[N, G]$. In particular, $[N, G] \subseteq {\bf C}_{[N, G]}(x)$. It follows that $x \in \zent N=1$, contradicting the hypothesis that $x \neq 1$. Therefore, ${\bf C}_{[N, G]}(x)=1$ for all $x \in {\bf N}_N(Q) \backslash \{1\}$. Then we have that $N$ is a Frobenius group with Frobenius kernel $[N, G]$. 

Now, if $\psi \in \Irr([N, G])$, given that $[N, G]$ is abelian, then $\psi$ is linear and $\psi(x) \neq 0$ for any $x \in [N, G]$. Since $[N, G]$ is nontrivial by hypothesis, the vanishing-off subgroup satisfies
$$\textnormal{\bf V}(\psi)=\langle g \in G \mid \psi(g) \neq 0 \rangle > 1$$ 
and $|N:[N, G]| \in \cd N$ (see Theorem (12.4) of \cite{Is}). Therefore, $\cd N$ is exactly $\{1,m\}$ with $m=|N:[N, G]|$ (in particular $|\Mcd G N|\geq 2$) and this situation is (b.2).

Finally, suppose that $\Mcd G N = \{1,m\}$ for some integer $m$. In this case, we may partition $\irr N$ as
$$\irr N = \lin G N \sqcup \Delta \sqcup \Xi,$$
where $\Delta$ is a set of representatives of the orbits of the linear characters of $N$ which are not $G$-invariant, and $\Xi$ is a set of representatives of the orbits of the nonlinear irreducible characters of $N$. Note that, since $\cd N = \{1,m\}$, every nonlinear $\theta \in \Irr(N)$ verifies that $f=\widehat{\theta}(1)=t_\theta \theta(1)$ for the same value $t_\theta=f/|N:[N, G]|$. Recall, too, that $|N:[N, G]|$ is the number of linear $G$-invariant characters of $N$.
Then
$$|N|=\sum_{\theta \in \irr N} \theta(1)^2 = \sum_{\theta \in \lin G N} \theta(1)^2 + \sum_{\theta \in \Delta} t_\theta \theta(1)^2 + \sum_{\theta \in \Xi} t_\theta \theta(1)^2,$$
that is,
$$|N|=  |N:[N, G]| + \sum_{\theta \in \Delta} f + \sum_{\theta \in \Xi} \frac{\widehat{\theta}(1)^2}{t_\theta} = |N:[N,G]| + \sum_{\theta \in \Delta} f + \sum_{\theta \in \Xi} f |N:[N, G]| .$$

But this implies that
$$|N|=|N:[N, G]| + K f$$
for some positive integer $K$. Since $p$ divides $|N|$ but not $|N:[N, G]|$, we deduce that $p$ cannot divide $f$ and the result follows.
\end{proof}

\begin{remark}The groups provided in Example \ref{ej 3} show that every case described in Proposition \ref{12.3 mcd} can actually be realized.\end{remark}

\medskip

\begin{proof}[Proof of Theorem C] We may assume $[N,G]>1$, otherwise it is clearly abelian. By theorem A, we know that $N$ is solvable. Let $K$ be a normal subgroup of $G$ contained in $N$ and maximal with respect to the property that $[N,G]$ is not contained in $K$ (it is possible that $K=1$). Denote $T \unlhd G$ such that $T/K=[N,G]K/K$. Then
$$1 \neq T/K=[N/K, G/K] \unlhd G/K.$$
If $1 \neq M/K $ is a minimal normal subgroup of $G/K$ contained in $N/K$, then $M/K$ is abelian and by maximality of $K$ we have that $[N,G] \subseteq M$. We conclude that $M/K=T/K$ and $T/K$ is the unique minimal normal subgroup of $G/K$ contained in $N/K$. By Lemma \ref{induccion} (i), we have that $\Mcd {G/K} {N/K} \subseteq \Mcd G N= \{1, f\}.$ If $\Mcd {G/K} {N/K} = \{1\}$, then $N/K \subseteq \zent {G/K}$ and then $[N,G] \subseteq K$, a contradiction. Therefore, $\Mcd {G/K} {N/K} = \{1, f\}$. By Proposition \ref{12.3 mcd}, we deduce that either $N/K$ is a $p$-group, for some prime $p$, or $N/K$ is Frobenius where $T/K$ is the Frobenius kernel of $N/K$. In both cases, $T/K$ is an elementary abelian $p$-group and $p$ is a prime divisor of order of the subgroup $[N,G]$. Moreover, in the second case, $p$ does not divide $f$.

Now we suppose $N/K$ is a $p$-group.  We have that $[T/K, G/K]= 1$ or $[T/K, G/K]=T/K$. Assume $[T/K, G/K]= 1$, and $T /K$ is central in $G/K$. If every irreducible character of $N/K$ is a linear $G/K$-invariant character, then $N/K$ is central in $G/K$, and this is a contradiction. Therefore, there exists $\theta \in \irr {N/K} \setminus \lin {G/K} {N/K}$.  Let $Q$ be a Sylow $q$-subgroup of $G$, with $q$ a prime distinct from $p$. By coprime action,
$$N/K = [N/K, QK/K]{\bf C}_{N/K}(QK/K)={\bf C}_{N/K}(QK/K).$$
Therefore, $QK/K \subseteq I_{G/K}(\theta)$, for all $q \neq p$. In this case, $|G/K:I_{G/K}(\theta)|$ and $\theta(1)$ are $p$-numbers, thus $f$ is a $p$-power.

Suppose instead $[T/K, G/K]=T/K$. Then $1_{T/K}$ is the unique $G/K$-invariant character of $T/K$. Then
$$|T/K|=1 + \sum_{1_{T/K} \neq \theta \in \irr {T/K}} \theta(1)^2 = 1 + s f,$$
for some integer $s$, and thus $p$ does not divide $f$.

Let $\Omega = \{p=p_1, p_2, \ldots, p_r\}$ the set of prime divisors of the order of $[N,G]$. We may assume that $N/K$ is a $p$-group or that $N/K$ is a Frobenius group, with $p \in \Omega$, and $p$ does not divide $f$. Then $p$ does not divide $\theta(1)$, for every $\theta \in \irr N$. By It\^o-Michler theorem, there exists $P_1$ an abelian normal Sylow $p$-subgroup of $N$, and then $P_1 \unlhd G$. If $[N,G]$ is contained in $P_1$, then it is abelian.

Suppose otherwise that $[N,G]$ is not contained in $P_1$. Let then $K_1 \unlhd G$ be maximal such that $[N,G]$ is not contained in $K_1$ and $P_1 \subseteq K_1 \subset N$. By applying to $K_1$ an argument similar to that used for $K$, we may take a prime $p \neq p_2 \in \Omega$, and $P_2$ a Sylow $p_2$-subgroup of $N$ such that $P_2 \unlhd G$ and $P_2$ is abelian. In this case, either $[N,G]$ is abelian, or $[N,G]$ is not contained in the abelian subgroup $\subseteq P_1 \times P_2$. By applying a recursive argument to the primes of $\Omega$, it follows in finitely many steps that $[N, G]$ is abelian.

In conclusion, we have that either $f$ is a power of $p$ or $[N,G]$ is abelian, and the first part of Theorem C follows.

Finally, suppose that $f$ is a power of $p$. Then it follows from Proposition \ref{p grupo hipercentral} that $N$ is hypercentral in $G$ and $[N,G] \subseteq \Phi(G)$.
\end{proof}

Let us examine the Conjecture B in the case where $f=p$, with $p$ a prime number. We point out that in this case $N'$ is abelian because $\textnormal{cd}(N) \subseteq \{1,p\}$. However, $N'$ does not always coincide with the commutator subgroup 
$[N, G]$. For example, if we consider in GAP the group $G:=\texttt{SmallGroup(216,26)}$, then there exists $C_9 \rtimes C_4 \cong N \unlhd G$ such that $\Mcd G N = \{1,2\}$ and $N' \cong C_{9}$, but $[N,G] \cong C_{18}$. We show that this last commutator is always abelian, as determined by Theorem D.

\begin{proof}[Proof of Theorem D] 
Suppose the result is false and choose $(G, N)$ to be a minimal counterexample with respect to $|G| + |N|$. By Theorem A, we know that $N$ is solvable. Moreover, every nonlinear irreducible character $\theta$ of $N$ is a $G$-invariant character of degree equal to $p$, so $\cd{N}=\{1,p\}$. By Thompson's theorem, we have that $N$ has a normal $p$-complement $A$. If $\lambda \in \Irr(A)$ and $\theta \in \Irr(N \mid \lambda)$, then $\lambda(1)$ divides $\theta(1)$ which divides $p$. Thus $\lambda$ is linear and $A$ is an abelian normal subgroup of $G$.

Let $L$ be a minimal normal subgroup of $G$ contained in $N$. Since $N$ is solvable, we have that $L$ is abelian. By Lemma \ref{induccion}, we have that $\Mcd{G/L}{N/L} \subseteq \{1,p\}$. If $\Mcd{G/L}{N/L}=\{1\}$, then $N/L \subseteq \zent{G/L}$, and thus $[N, G] \subseteq L$, a contradiction. If $\Mcd{G/L}{N/L}=\{1,p\}$, then, by minimality of the counterexample, we have that $[N/L,G/L]$ is abelian. That is, $[N, G]/(L \cap [N, G])$ is abelian, hence $[N, G]' \subseteq L$. By minimality of $L$, we deduce that $[N, G]'=L$ is the unique minimal normal subgroup of $G$ contained in $N$. Then either $\oh p N =1$ or $\oh {p'} N=1$.
We distinguish two possible cases: either $A=1$ and $N$ is a $p$-group or $A\ne 1$ and $\oh p N=1$.

\textit{Case 1:} Suppose that $N$ is a $p$-group. By Proposition \ref{p grupo hipercentral}, $N$ is hypercentral in $G$, and then $[N,G] < N$. Since $\cd{N}=\{1,p\}$, it follows that  $\Mcd{G}{[N,G]}=\{1,p\}=\cd{[N,G]}=\{1,p\}$. By the minimality of $(G,N)$, we have that $[[N, G],G]$ is abelian. Moreover, $L$ is also a hypercentral subgroup of $G$, hence $1=[L,G]< L$ and $L$ is a central subgroup of $G$ of order $p$. 
 
 Let $\lambda \in \irr{[N,G]} \setminus {\rm Lin}([N,G])$ and $\widehat{\lambda} \in {\rm Min}_G([N,G])$. If $1 \neq {\rm ker}( \widehat{\lambda}) \unlhd [N,G]$, then $L \subseteq {\rm ker}(\widehat{\lambda}) \subseteq {\rm ker}(\lambda)$, and then $\lambda$ is linear, a contradiction. Hence ${\rm ker}(\widehat{\lambda})=1$. Now, let $g \in [N, G] \setminus \text{\bf Z}([N, G]))$. There exists an element $b \in [N, G]$ such that $z=[g,b] \neq 1$. We also have that $z \in L \subseteq \text{\bf Z}(G)$. Let $\mathcal{R}$ be a representation of $[N,G]$ affording $\lambda$. Then $\mathcal{R}(gz)=\mathcal{R}(b^{-1}gb)=\mathcal{R}(g)\mathcal{R}(z)$. Since $z \in \text{\bf Z}([N,G])$, we have that $\mathcal{R}(b^{-1}gb)= \alpha \mathcal{R}(g)$, for some $\alpha \in \mathbb{C}$. By taking traces, we obtain that either $\alpha = 1$ or $\lambda(g)=0$. The first case is not possible, as it would imply that $z \in \ker \lambda$ and then
$$z \in \left(\bigcap _{x \in N} \ker{\lambda^x} \right)\cap L= {\rm ker}( \widehat{\lambda}) \cap L=1,$$
a contradiction. Therefore, $\lambda(g)=0$, for every $g \in [N, G] \setminus \text{\bf Z}([N,G])$. Since $\lambda_{\text{\bf Z}([N, G])}=\lambda(1)\beta$ with $\beta \in \text{Lin}(\text{\bf Z}([N, G]))$, by Lemma (2.29) of \cite{Is}, it follows that
$$\lambda(1)^2=[\lambda_{\text{\bf Z}([N, G])}, \lambda_{\text{\bf Z}([N, G])}]=|[N, G]:\text{\bf Z}([N, G])|\,[\lambda, \lambda]=|[N, G]:\text{\bf Z}([N, G])|.$$
Hence,
$$\text{cd}([N, G])=\left\{1, \sqrt{|[N, G]:\text{\bf Z}([N, G])|}\right\}=\{1,p\}.$$
Therefore, $|[N, G]:\text{\bf Z}([N, G])|=p^2$.

We now proceed to prove that ${\bf Z}([N, G])$ is central in $G$.  We take $\mu \in \Irr({\bf Z}([N, G]))$ and $\lambda \in \Irr([N, G]\ |\ \mu)$ such that $\lambda(1)=p$. We have that $\lambda$ and $\mu$ are $G$-invariant characters of $[N,G]$ and ${\bf Z}([N, G])$, respectively. Then $ [{\bf Z}([N, G]),G]\subseteq \ker{\lambda}\cap {\bf Z}([N, G])=\ker\mu \subseteq \ker\lambda$. Suppose that $[{\bf Z}([N, G]),G] \neq 1$, then $L \subseteq [{\bf Z}([N, G]),G]\subseteq  \ker\lambda$ and $\lambda$ is linear, a contradiction. Then ${\bf Z}([N, G]) \subseteq {\bf Z}(G)$, as claimed.

We show that $N'$ is non-central in $G$. We assume that $N' \subseteq {\bf Z}(G)$. As previously seen, $[[N, G],G] = [G, [N,G]]$ is abelian. Thus, 
$$[G,[N, G],[[N, G],G]]=1$$
and, since $N'$ is central in $G$, we have
$$[[N, G],[[N, G],G],G]=1.$$
According to the three subgroups lemma, we obtain that
$$[[N, G],G],G,[N, G]]=1$$
and 
$$[[[N, G],G],G] \subseteq {\bf Z}([N, G]) \subseteq {\bf Z}(G) \cap N.$$
Consequently, 
$$[[[N, G],G],G,N]=1$$
and, as $[N, [[N,G],G]] \subseteq N' \subseteq {\bf Z}(G)$, we have that $[N, [[N, G],G],G]=1$. Using the three subgroups lemma again, we obtain that
$$[G,N,[[N, G],G]]=1$$
and $[[N, G],G] \subseteq {\bf Z}([G, N])={\bf Z}([N, G]) \subseteq {\bf Z}(G) \cap N$. Therefore
$$[G, [N, G],N]=1.$$
As $N' \subseteq {\bf Z}(G)$, we also have that
$$[[N, G],N,G]=1.$$
Hence, by the three subgroups lemma, we obtain that
$$L=[[N, G],[N, G]]=[N,G,[N, G]]=1,$$
which leads to a contradiction.

Therefore, we can consider that $N'$ is not central in $G$. Since $\cd N=\{1,p\}$, by Corollary (12.6) in \cite{Is}, we have that $N'$ is abelian. Let $U=N'{\bf Z}([N, G])$, which is an abelian subgroup of $[N,G]$. If $U={\bf Z}([N, G])$, then $N' \subseteq {\bf Z}(G)$, and this is not possible. Therefore, $U \neq {\bf Z}([N, G])$ and $|U: {\bf Z}([N, G])|=p$.

Next, we prove that $[U,G]=[N', G]\neq 1$ is central in $G$. For that purpose, we claim that $U$ is a $G$-invariant $nMI$-subgroup with $|\Lcd G U |\leq 2$ (according to the notation given at the end of Section 2). Let $\eta \in \irr U$. If $\eta \in \mathrm{Lin}_G(U)$, considering $H_{\eta}=U$ and $\lambda_{\eta}=\eta$, it follows that $(H_{\eta},\lambda_{\eta})$ is a linear $G$-invariant character pair with respect to $\eta$. Suppose that $\eta$ is not a $G$-invariant character. Then $\rho=\eta_{{\bf Z}([N, G])} \in {\rm Lin}_{G}({\bf Z}([N, G]))$ and $\rho^U$ is  exactly a sum of $G$-conjugates of $\eta$. Hence, $\irr {G \mid \eta}=\irr {G \mid \rho}$ and $(H_{\eta}={\bf Z}([N, G]),\lambda_{\eta}=\rho)$ is a linear $G$-invariant character pair with respect to $\eta$.  Since  $ \Lcd G U = \Mcd G U \subseteq \{1,p\}$, by Theorem \ref{nMI} we have that $[U,G,G]=1$ and $[U,G]$ is central in $G$.  In this situation, we have that $[N',G,G]=1$ and $[G,N',G]=1$. By the three subgroups lemma, it follows that $[G,G,N']=1$ and then $[G,N,N']=1$, which implies that $N' \subseteq {\bf Z}([N, G]) \subseteq {\bf Z}(G)$, a contradiction.

\bigskip

{\it Case 2:} Suppose that $A\ne 1$ and $\oh p N=1$. We have that $N=AP$, where $P$ is a Sylow $p$-subgroup of $N$. Since $\cd{N}=\{1, p\}$, by Theorem (12.5) of \cite{Is}, we have that $N$ has an abelian normal subgroup $T \unlhd N$ of index $p=|N:T|$.
By coprime action, we have
$$A= [A,P] \times  {\bf C}_A(P)$$
and $G=N{\bf N}_G(P)=A{\bf N}_G(P)$ by Frattini's argument, thus $[A,P]$ and ${\bf C}_A(P)$ are normal in $G$. Then, $L= [N,G]' \subset [A, P]$ and $C_A(P)=1$. We deduce that $A=[A, P]=[A,G]$ and $[N,G]=AP_0$, where $P_0=P \cap [N,G] \neq 1$. Moreover, $C_{P}(A) \unlhd G$ and, as  $\oh p N=1$, we conclude $C_{P}(A)=1$, for all Sylow $p$-subgroups $P$ of $N$. As a consequence, the Sylow $p$-subgroup of $T$ is trivial and $T=A$. Hence, $N=A\langle y \rangle$ with $y^p=1$. We have that $P=\langle y \rangle$, $N=AP=AP_0=[N,G]$, thus the principal character of $N$ is the unique linear $G$-invariant character of $N$. Therefore, since $\Mcd{G}{N}=\{1, p\}$, we have
$$|A|p=|N|= 1 +\sum_{\theta \notin {\rm Lin}_G(N)} \theta(1)^2=1 + k p,$$
for some integer $k$, a contradiction. This fact completes the proof of Theorem D.
\end{proof}

\bigskip

\section{Some examples} \label{sección ejemplos}

In this final section, we give a few examples to illustrate the main results and show the necessity of some of the hypotheses. The notation of the groups and their normal subgroups all refer to the Small Group library in GAP (see \cite{GAP}).

\medskip

We first present some considerations regarding Lemma \ref{centro} and Proposition \ref{burnside}, as well as the properties of minimal invariant characters.

\begin{example} \label{ej 1}
 Take $G=D_{18}$ and $N=C_3$. Let $\theta=1_N$ and $\chi$ be the only irreducible character of $G$ of degree $2$ lying over $\theta$. Let also $g \in N \backslash \{1\}$. Then $|g^G|=2$, and thus $(\chi(1),|g^G|)\neq 1$, but $\theta$ is $G$-invariant, so $(\widehat{\theta}(1),|g^G|)=(\theta(1),|g^G|)=1$. Then, as stated in Remark \ref{remark burnside}, the hypothesis $(\widehat{\theta}(1),|g^G|)=1$ in Proposition \ref{burnside} does not imply that $(\chi(1),|g^G|)=1$, so Burnside's theorem cannot be applied directly to $\chi$ and $G$.

Now, denote ${\theta_1,\theta_2}=\irr N \backslash \{1_N\}$. We have that $\widehat{\theta_1}=\widehat{\theta_2}=\theta_1+\theta_2$, with $\zent{\widehat{\theta_1}}=\{1\}$. We have as well that $\widehat{\theta_1}(g)=-1$. Then, although it is clearly true that $(\theta_1(1),|g^N|)=1$, we do not have that $g \in \zent{\widehat{\theta_1}}$ nor $\widehat{\theta_1}(g)=0$. Therefore, Proposition \ref{burnside} needs stronger hypotheses than applying Burnside's theorem to $N$ does.

Note as well that $\zent {\theta_1} = \zent {\theta_2}= N$. Thus $\zent {\theta_1} \cap \zent {\theta_2}= N$ and the inclusion $\zent {\widehat{\theta_1}} \subseteq \bigcap_{x \in G} \zent{\theta_1}^x$ given by Lemma \ref{centro} is strict in this case.
\end{example}

In the remaining examples, we refer to results presented in Section \ref{sección main results}.

\begin{example} \label{ej 2} The concern of Proposition \ref{p grupo hipercentral} and Corollary \ref{corolario p grupo hipercentral} are normal subgroups $N$ which are $p$-groups. This hypothesis cannot be easily weakened: take, for instance, $G=D_{12}$ and $N=S_3 \unlhd G$. Then $\Mcd G N = \{1,2\}$ and $N'=[N,G]=[N,G,G]$, thus $N$ is not nilpotent nor hypercentral in $G$. In fact, even when $\Mcd G N=\{1,p^a\}$ and $N$ is nilpotent, it does not need be hypercentral if it is not a $p$-group. To see this, let $G=\texttt{SmallGroup(36,3)}$ and $N\cong C_6 \times C_2$. Then $\Mcd G N = \{1,3\}$ but $[N,G,G]=[N,G]$, so $N$ cannot be hypercentral in $G$.

Suppose now $N$ is a $p$-group. Proposition \ref{p grupo hipercentral} does not hold either if $\Mcd G N$ contains an integer that is not a power of $p$. Let indeed $G=A_4$ and $N=C_2 \times C_2$. Then $\Mcd G N = \{1,3\}$ and $N=[N,G]$.

When $N$ satisfies the hypotheses of Corollary \ref{corolario p grupo hipercentral} and $N=G$, that is, when $G$ is a $p$-group and $|\cd G|=2$, then the proof of Theorem (12.14) in \cite{Is} (using Theorem (12.5) in the same book) shows that the nilpotency class of $G$ is at most $3$. Within the context of minimal $G$-invariant characters, one may wonder whether the hypercentral $G$-length of a proper normal subgroup $N$ is also bounded by $3$. However, if we consider $G=\texttt{SmallGroup(256,525)}$ and  $N$ any normal subgroup of $G$ of size $128$, then $\Mcd G N = \{1,2\}$ and the hypercentral $G$-length of $N$ is $6$.
\end{example}

\begin{example}\label{ej 3} Lastly, we list a series of groups satisfying every possible situation given by Proposition \ref{12.3 mcd}.
\vspace*{-10pt}
\newline \begin{itemize}
\item[(i)] Let us take $G=\texttt{SmallGroup(108,21)}$ and $N$ the only normal subgroup of $G$ of size $4$ such that $[N,G]$ is the unique minimal normal subgroup of $G$ contained in $N$. In this case, $N \cong C_2 \times C_2$ is an elementary abelian $2$-group, and $\zent G \cap N = 1$. This is situation (a).
\item[(ii)] Take again $G=\texttt{SmallGroup(108,21)}$ and $N$ the only normal subgroup of $G$ of size $9$ such that $[N, G]$ is the unique minimal normal subgroup of $G$ contained in $N$. Then $N \cong C_3 \times C_3$ and $N/(\zent G \cap N) \cong C_3$. This is again situation (a), but in this case we have $1<\zent G \cap N \cong C_3$.
\item[(iii)] Take $G= C_5 \times D_8$ and $N= D_8$. Then $\zent N = \zent G \cap N = C_2$ and $\cd N = \{1,2\}$. This is situation (b.1).
\item[(iv)] Take $G=\texttt{SmallGroup(32,50)}$. Then there exists $(C_4 \times C_2) \rtimes C_2 \cong N \unlhd G$ such that $[N,G]$ is the unique minimal normal subgroup of $G$ contained in $N$. Moreover, $\cd N = \{1,2\}$, $\zent G \cap N \cong C_2$ and $N/(\zent G \cap N) \cong C_2 \times C_2 \times C_2$, but $C_4 \cong \zent N > \zent G \cap  N$. This is situation (b.1).
\item[(v)] Take $G=D_{12}$ and $N=S_3$. Then $N$ is a Frobenius group with kernel $[N,G]=C_3$ and $\Mcd G N = \{1,2\}$. This is situation (b.2).
    \item[(vi)] Take $G=\texttt{SmallGroup(200,43)}$ and $N$ the largest normal subgroup of $G$ such that $[N, G]$ is the unique minimal normal subgroup of $G$ contained in $N$. We have that $N \cong (C_5 \times C_5) \rtimes C_2$. Then $N$ is a Frobenius group with kernel $[N, G] \cong C_5 \times C_5$, and $\Mcd G N = \{1,4,8\}$. This is, once again, situation (b.2). Note that, in this case, the additional hypothesis $|\Mcd G N|=2$ is not satisfied.
\end{itemize}
\end{example}

\noindent \textbf{Acknowledgment} The authors would like to thank Zeinab Akhlaghi for many helpful conversations on this topic.

\bigskip

\noindent \footnotesize{\textsc{Mar\'{\i}a Jos\'e Felipe, Institut Universitari de Matem\`atica Pura i Aplicada, \newline
Universitat Polit\`ecnica de Val\`encia, 
Val\`encia, Spain.} \newline
\texttt{mfelipe@mat.upv.es}}

\bigskip

\noindent \footnotesize{\textsc{Iris Gilabert, Institut Universitari de Matem\`atica Pura i Aplicada, \newline
Universitat Polit\`ecnica de Val\`encia, 
Val\`encia, Spain.} \newline
\texttt{igilman@posgrado.upv.es}}

\bigskip

\noindent \footnotesize{\textsc{Lucia Sanus, Departament de Matem\`atiques, Facultat de
 Matem\`atiques, \newline
Universitat de Val\`encia,
46100 Burjassot, Val\`encia, Spain.} \newline
\texttt{lucia.sanus@uv.es}}

\end{document}